\documentclass[10pt,reqno]{amsart}

\usepackage[T1]{fontenc}
\usepackage{lmodern}
\usepackage{microtype}
\usepackage{mathtools,amssymb}
\usepackage{booktabs,tabularx,array}
\usepackage[table]{xcolor}
\usepackage{enumitem}
\usepackage[colorlinks=true,linkcolor=blue!50!black,citecolor=blue!50!black,
  urlcolor=blue!50!black,
  pdftitle={A minimal HRT counterexample at the sharp Gelfand--Shilov threshold},
  pdfauthor={Vignon Oussa}]{hyperref}
\usepackage[capitalize,noabbrev]{cleveref}

\allowdisplaybreaks
\numberwithin{equation}{section}
\newtheorem{theorem}{Theorem}[section]
\newtheorem{proposition}[theorem]{Proposition}
\newtheorem{lemma}[theorem]{Lemma}
\newtheorem{corollary}[theorem]{Corollary}
\theoremstyle{definition}
\newtheorem{remark}[theorem]{Remark}

\newcommand{\R}{\mathbb R}
\newcommand{\C}{\mathbb C}
\newcommand{\Q}{\mathbb Q}
\newcommand{\Z}{\mathbb Z}
\newcommand{\T}{\mathbb T}
\newcommand{\cS}{\mathcal S}
\newcommand{\cZ}{\mathcal Z}
\newcommand{\cZtwo}{\mathcal Z_2}
\newcommand{\ii}{\mathrm i}
\newcommand{\e}{\mathrm e}
\newcommand{\abs}[1]{\lvert#1\rvert}
\newcommand{\norm}[1]{\lVert#1\rVert}
\newcommand{\spn}{\operatorname{span}}

\definecolor{HRTHeader}{HTML}{DCE6F1}
\definecolor{HRTRegime}{HTML}{F2F2F2}
\definecolor{HRTPositive}{HTML}{E2F0D9}
\definecolor{HRTPartial}{HTML}{FFF2CC}
\definecolor{HRTOpen}{HTML}{E7E6E6}
\definecolor{HRTCounterexample}{HTML}{F4CCCC}
\definecolor{HRTFramework}{HTML}{D9EAF7}
\definecolor{HRTSharp}{HTML}{E4DFEC}

\title[Sharp minimal HRT counterexample]
  {A minimal HRT counterexample at the sharp Gelfand--Shilov threshold}
\author{Vignon Oussa}
\address{Department of Mathematics, Bridgewater State University,
Bridgewater, Massachusetts, USA}
\email{voussa@bridgew.edu}
\date{August 16, 2026}

\begin{document}

\begin{abstract}
In a remarkable breakthrough, Faulhuber, Petersen, van Velthoven, and
Voigtlaender disproved the Heil--Ramanathan--Topiwala conjecture by constructing
twelve linearly dependent time--frequency shifts of a Schwartz function.  The
following day, the author compressed their vector-Zak and cohomological
architecture to an explicit intrinsically subcritical four-point configuration,
thereby attaining the minimum possible cardinality.  Soon thereafter, Jasper
and Mixon gave a transparent analytic refinement of the breakthrough mechanism
and proved, for every Diophantine pair, that HRT counterexamples may be realized
with windows in the Roumieu Gelfand--Shilov class $\cS_1^1(\R)$; their theorem
establishes the endpoint regularity phenomenon in broad arithmetic generality
while leaving support cardinality unrestricted.  The purpose of the present
paper is to show that these two sharp advances can be attained simultaneously.
We prove that the same explicit four points, coefficients, and matrix cocycle
admit a nonzero window in $\cS_1^1(\R)$.  Moreover, if
$0\ne f\in\cS_s^\sigma(\R)$ and $s<1$ or $\sigma<1$, then every finite system
of distinct time--frequency shifts of $f$ is linearly independent.  Thus four
is the least possible cardinality and $(1,1)$ is the sharp coordinatewise
threshold within the Roumieu Gelfand--Shilov scale.
\end{abstract}

\maketitle

\section{Statement of the sharp result}

For $z=(x,\omega)\in\R^2$, let
\begin{equation}\label{eq:weyl}
 (\rho(x,\omega)f)(t)=\e^{2\pi\ii\omega(t-x/2)}f(t-x)
\end{equation}
be the Weyl shift on $L^2(\R)$.  For $s,\sigma>0$, the Roumieu class
$\cS_s^\sigma(\R)$ consists of the functions $f\in C^\infty(\R)$ for which
there are $C,h>0$ satisfying
\begin{equation}\label{eq:GSdef}
 \sup_{t\in\R}\abs{t^m f^{(n)}(t)}
 \le Ch^{m+n}m!^s n!^\sigma,
 \qquad m,n\ge0.
\end{equation}
Jasper and Mixon recently gave a clear strip-holomorphic treatment of the
breakthrough mechanism and proved that, for every Diophantine pair, one can
construct a finite HRT dependence with a nonzero window in $\cS_1^1(\R)$
\cite[Theorem~2]{JasperMixon2026}.  Their theorem is broad in arithmetic scope
and addresses endpoint regularity rather than support minimization.  The
theorem below is complementary and support-sharp: it places an $\cS_1^1$
window on the explicit four-point configuration from
\cite{OussaCounterexample2026}, without changing its coefficients or matrix
cocycle.

Set
\begin{equation}\label{eq:parameters}
 \vartheta=\sqrt[3]{2},\qquad
 \alpha=\vartheta-1,\qquad
 \beta=\vartheta^2-1,\qquad
 \zeta=(\alpha,\beta/2).
\end{equation}

\begin{theorem}[Sharp four-point counterexample]\label{thm:main}
There exist $0\ne g\in\cS_1^1(\R)$ and $\lambda\in\C^*$ such that
\begin{equation}\label{eq:eigen}
 \left[I+\frac35\rho(1,0)+\frac35\rho(0,1/2)\right]\rho(\zeta)g
 =\lambda g.
\end{equation}
Equivalently,
\begin{align}\label{eq:dependence}
 -\lambda g+\rho(\alpha,\beta/2)g
 &+\frac35\e^{-\pi\ii\beta/2}\rho(1+\alpha,\beta/2)g \notag\\
 &+\frac35\e^{\pi\ii\alpha/2}
 \rho\bigl(\alpha,(1+\beta)/2\bigr)g=0.
\end{align}
The four points are pairwise distinct, no three are collinear, and the four
absolute symplectic triangle determinants are
\begin{equation}\label{eq:determinants}
 \frac{\beta}{2},\qquad \frac{\alpha}{2},\qquad
 \frac{1+\alpha+\beta}{2},\qquad \frac12,
\end{equation}
all strictly between $0$ and $1$.
\end{theorem}

\begin{corollary}[Two simultaneous sharp thresholds]\label{cor:sharp}
A finite linearly dependent Gabor system with a nonzero window in
$\cS_s^\sigma(\R)$ exists if and only if $s\ge1$ and $\sigma\ge1$.
Moreover, among all nonzero $L^2$ windows, four is the minimum cardinality of
a dependent finite Gabor system.
\end{corollary}

The assertion is Roumieu.  We do not claim the Beurling endpoint
$\Sigma_1^1$, which would require the estimates in \eqref{eq:GSdef} for every
$h>0$.

\section{Chronology and the two-stage trichotomy}

Every noncollinear four-point set admits a marked presentation
\begin{equation}\label{eq:marked}
 \Lambda=\{0,u,v,q\},\qquad q=au+bv,\qquad [u,v]\ne0.
\end{equation}
The first classification stage is geometric,
\begin{equation}\label{eq:delta}
 \delta=\abs{[u,v]}:\qquad \delta>1,\quad \delta=1,\quad \delta<1,
\end{equation}
called supercritical, critical, and subcritical.  The second is arithmetic,
\begin{equation}\label{eq:rank}
 \varrho=\dim_{\Q}\spn_{\Q}\{1,a,b\}\in\{1,2,3\},
\end{equation}
corresponding to rational, mixed, and maximally irrational rogue coordinates.
The marking matters: the four possible triangle covolumes are
$\delta$, $\abs a\delta$, $\abs b\delta$, and
$\abs{1-a-b}\delta$.  Rational rank is unchanged by every admissible re-marking: for example, if
$b\ne0$, the basis $(u,q)$ gives the old point $v$ the coordinates
$(-a/b,1/b)$, and multiplication by $b$ identifies its $\Q$-span with
$\spn_{\Q}\{1,a,b\}$; the other re-markings are analogous.  The $(1,3)$
and $(2,2)$ labels are a separate affine-geometric classification, not
rational-rank labels.

\begin{table}[t]
\caption{Milestones leading to the sharp four-point result.  The online
research record \cite{OussaResearchRecord2026} supplies source links; the
mathematical claims are attributed to the cited papers and preprints.}
\label{tab:timeline}
\scriptsize
\setlength{\tabcolsep}{3.5pt}
\renewcommand{\arraystretch}{1.08}
\begin{tabularx}{\textwidth}{@{}>{\raggedright\arraybackslash}p{0.105\textwidth}
 >{\raggedright\arraybackslash}p{0.245\textwidth}
 >{\raggedright\arraybackslash}X@{}}
\toprule
Date & Result & Logical role\\
\midrule
June 6, 2024 & \cellcolor{HRTFramework}\textbf{ICERM framework}
\cite{OussaICERM2024} & Zak zeros and rational dimension are presented as
organizing invariants; this is the public starting point of the later
configuration program.\\
June 2025 & \cellcolor{HRTPartial}\textbf{One-rogue partial theorem}
\cite{OkoudjouOussa2025} & For an integer-lattice block plus one rogue point,
rank three is excluded for $W_0$/Schwartz windows; rank two is reduced to a
nonfinite Zak zero set.\\
Mar.--May 2026 & \cellcolor{HRTPartial}\textbf{Orbit rigidity}
\cite{OussaTrichotomy2026,OussaInfiniteNonDense2026} & The finite and dense
orbit branches are excluded, and the infinite proper branch is constrained
and then closed in the stated mixed-integer Schwartz scope.\\
April 23, 2026 & \cellcolor{HRTPositive}\textbf{Positive cells}
\cite{OussaFourPoint2026} & All rational-coordinate cases and the full
supercritical rank-three four-point cell are proved for every nonzero
$L^2$ window.\\
June 27, 2026 & \cellcolor{HRTPositive}\textbf{Critical rank two}
\cite{OussaCriticalMixed2026} & The complete marked cell
$(\delta,\varrho)=(1,2)$ is proved for every nonzero $L^2$ window.\\
July 18, 2026 & \cellcolor{HRTPositive}\textbf{Critical rank three}
\cite{OussaCriticalRankThree2026} & The complete marked cell
$(\delta,\varrho)=(1,3)$ is proved for every nonzero $L^2$ window.\\
August 5, 2026 & \cellcolor{HRTCounterexample}\textbf{HRT disproved}
\cite{FPPV2026} & The twelve-point Schwartz counterexample introduces the
vector-Zak, invariant-line, winding, and cohomological mechanism.\\
August 6, 2026 & \cellcolor{HRTCounterexample}\textbf{Four-point example}
\cite{OussaCounterexample2026} & The mechanism is compressed to the minimum
cardinality and placed intrinsically in the subcritical rank-three cell.\\
August 13, 2026 & \cellcolor{HRTSharp}\textbf{Endpoint regularity}
\cite{JasperMixon2026} & Jasper and Mixon give a transparent strip-holomorphic
construction of an $\cS_1^1$ counterexample for every Diophantine pair,
establishing the endpoint phenomenon in broad arithmetic generality while
leaving support cardinality unrestricted.\\
August 16, 2026 & \cellcolor{HRTSharp}\textbf{Simultaneous sharp endpoint} &
The present refinement places an $\cS_1^1$ window on the explicit four-point
configuration, attaining the regularity and cardinality endpoints
simultaneously.\\
\bottomrule
\end{tabularx}
\end{table}

To read the status table, fix one marked cell and consider the universal
assertion that every configuration in that cell is HRT-independent for every
nonzero $L^2$ window.  Green means that this assertion is proved; amber means
that it remains open but substantial subclasses or obstructions are known;
gray means open without comparable cell-level structure; rose means that the
universal assertion is false because a counterexample exists.  Rose does not
mean that all dependent configurations in the cell are classified.

\begin{center}
\scriptsize
\colorbox{HRTPositive}{\strut\textbf{Established positive}}\quad
\colorbox{HRTPartial}{\strut\textbf{Open with partial results}}\quad
\colorbox{HRTOpen}{\strut\textbf{Open}}\quad
\colorbox{HRTCounterexample}{\strut\textbf{Counterexample exists}}
\end{center}

\begin{table}[t]
\caption{Status as of August 16, 2026 of the universal four-point assertion
in the marked covolume--rank classification.  Results announced in preprints
are identified by citation.}
\label{tab:cells}
\scriptsize
\setlength{\tabcolsep}{3pt}
\renewcommand{\arraystretch}{1.14}
\begin{tabularx}{\textwidth}{@{}>{\raggedright\arraybackslash}p{0.14\textwidth}
 >{\raggedright\arraybackslash}X>{\raggedright\arraybackslash}X
 >{\raggedright\arraybackslash}X@{}}
\toprule
\rowcolor{HRTHeader}
Geometric stage & $\varrho=1$ rational & $\varrho=2$ mixed &
$\varrho=3$ maximally irrational\\
\midrule
\cellcolor{HRTRegime}\textbf{$\delta>1$}\newline Supercritical &
\cellcolor{HRTPositive}\textbf{Established.} Rational refinement and
Linnell \cite{Linnell,OussaFourPoint2026}. &
\cellcolor{HRTPartial}\textbf{Partially open.} Strong cyclicity,
orbit-closure, and recurrence subclasses are known, but no full arbitrary
$L^2$ cell theorem. &
\cellcolor{HRTPositive}\textbf{Established.} Full cell for every nonzero
$L^2$ window \cite{OussaFourPoint2026}.\\
\addlinespace
\cellcolor{HRTRegime}\textbf{$\delta=1$}\newline Critical &
\cellcolor{HRTPositive}\textbf{Established.} Rational refinement and
Linnell. &
\cellcolor{HRTPositive}\textbf{Established.} Full cell
\cite{OussaCriticalMixed2026}. &
\cellcolor{HRTPositive}\textbf{Established.} Full cell
\cite{OussaCriticalRankThree2026}.\\
\addlinespace
\cellcolor{HRTRegime}\textbf{$\delta<1$}\newline Subcritical &
\cellcolor{HRTPositive}\textbf{Established.} Rational refinement and
Linnell. &
\cellcolor{HRTPartial}\textbf{Partially open.} Existing spectral,
cohomological, orbit, and fixed-window criteria do not settle the universal
cell. &
\cellcolor{HRTCounterexample}\textbf{Universal assertion false.}  The
intrinsically subcritical four-point example lies here
\cite{OussaCounterexample2026}; full classification remains open.\\
\bottomrule
\end{tabularx}
\end{table}

No cell is presently assigned the gray status.  The rank-three covolume
trichotomy is especially transparent: the universal positive assertion holds
for $\delta\ge1$ and fails for $\delta<1$.  The four values in
\eqref{eq:determinants} show that the example cannot be moved into a positive
row by changing the reference triangle.  The endpoint refinement proved below
is transverse to this table: it sharpens the window while leaving the marked
cell unchanged.

\section{The real-torus input}

The four points in \eqref{eq:dependence} are
\begin{equation}\label{eq:points}
 p_0=(0,0),\quad p_1=(\alpha,\beta/2),\quad
 p_2=(1+\alpha,\beta/2),\quad
 p_3=(\alpha,(1+\beta)/2).
\end{equation}

\begin{lemma}[Intrinsic subcriticality]\label{lem:geometry}
The determinants of the four triangles in \eqref{eq:points} are those in
\eqref{eq:determinants}; moreover $1,\alpha,\beta$ are linearly independent
over $\Q$.
\end{lemma}

\begin{proof}
Direct expansion gives, up to sign,
$\beta/2$, $\alpha/2$, $(1+\alpha+\beta)/2$, and $1/2$.  Since
$1<\vartheta<4/3$ and $1<\vartheta^2<5/3$, one has
$0<\alpha,\beta<1$ and $1+\alpha+\beta<2$.  A rational relation among
$1,\alpha,\beta$ would give a rational polynomial of degree at most two
vanishing at $\vartheta$, contrary to irreducibility of $X^3-2$.
\end{proof}

Define the scalar and two-component Zak transforms by
\begin{align}\label{eq:zak}
 (\cZ f)(x,\nu)&=\sum_{k\in\Z}f(x-k)\e^{2\pi\ii k\nu},\\
 (\cZtwo f)_{r+1}(x,\omega)&=2^{-1/2}(\cZ f)
 \left(x,\frac{\omega+r}{2}\right),\qquad r=0,1.\label{eq:vectorzak}
\end{align}
A vector-Zak section satisfies
\begin{equation}\label{eq:sewing}
 F(x+1,\omega)=U_1(\omega)F(x,\omega),\qquad
 F(x,\omega+1)=U_2F(x,\omega),
\end{equation}
where
\begin{equation}\label{eq:U}
 U_1(\omega)=\e^{\pi\ii\omega}D,\quad
 U_2=S,\quad
 D=\begin{pmatrix}1&0\\0&-1\end{pmatrix},\quad
 S=\begin{pmatrix}0&1\\1&0\end{pmatrix}.
\end{equation}
Let $\tau=(\alpha,\beta)$, $Tz=z-\tau$, $R=T^{-1}$, and
\begin{equation}\label{eq:AB}
 A(x,\omega)=I+\frac35\e^{-\pi\ii\omega}D
                 +\frac35\e^{\pi\ii x}S,
 \qquad
 B(x,\omega)=\eta(x)A(x,\omega),
\end{equation}
with $\eta(x)=\e^{\pi\ii\beta(x-\alpha/2)}$.  Direct Zak calculation gives
\begin{equation}\label{eq:conjugacy}
 \cZtwo\left[I+\frac35\rho(1,0)+\frac35\rho(0,1/2)\right]
 \rho(\zeta)\cZtwo^{-1}F(z)=B(z)F(Tz)
\end{equation}
and the twisted covariance
\begin{equation}\label{eq:covariance}
 B(z+e_j)=U_j(z)B(z)U_j(Tz)^{-1},\qquad j=1,2.
\end{equation}

The present note uses the following already certified real-torus output.  It
contains every computer-assisted premise used here; no new numerical
certificate is required.

\begin{theorem}[Real-torus package \cite{OussaCounterexample2026}]
\label{thm:realinput}
For the cocycle \eqref{eq:AB}:
\begin{enumerate}[label=\textup{(\roman*)},leftmargin=2em]
\item there is a smooth one-step invariant dominated line $E^s(z)$ and a
smooth nowhere-zero sewn vector $v_{\rm sm}(z)\in E^s(z)$;
\item there is a smooth periodic $q_{\rm sm}:\T^2\to\C^*$ such that
\begin{equation}\label{eq:realmult}
 B(z)v_{\rm sm}(Tz)=q_{\rm sm}(z)v_{\rm sm}(z),
\end{equation}
and both winding numbers of $q_{\rm sm}$ vanish;
\item the inverse sixteen-step projective graph transform preserves a uniform
closed tube about $E^s$ and contracts its projective slope metric by a factor
$\theta<1/50$;
\item for every $(m,n)\in\Z^2\setminus\{0\}$,
\begin{equation}\label{eq:small}
 \abs{1-\e^{-2\pi\ii(m\alpha+n\beta)}}
 \ge \frac{1}{3(1+\abs m+\abs n)^2}.
\end{equation}
\end{enumerate}
\end{theorem}

\section{Holomorphic refinement}

For $\rho>0$ write
\begin{equation}\label{eq:strip}
 \Omega_\rho=\{(x,\omega)\in\C^2:
 \abs{\Im x}<\rho,\ \abs{\Im\omega}<\rho\}.
\end{equation}
The entries of $A$ and $B$ are entire, and
\begin{equation}\label{eq:detstrip}
 \abs{\det A(z)}\ge1-\frac{18}{25}\e^{2\pi\rho}
 \qquad(z\in\Omega_\rho).
\end{equation}
Hence all finite forward and inverse blocks are holomorphic and invertible on
$\Omega_\rho$ whenever
\begin{equation}\label{eq:rhodet}
 0<\rho<\rho_{\det}:=\frac1{2\pi}\log\frac{25}{18}.
\end{equation}

\begin{lemma}[Entire approximation preserving sewing]\label{lem:approx}
Every smooth vector-Zak section can be approximated in every finite
$C^r$ norm on $[0,1]^2$ by entire vector-Zak sections satisfying
\eqref{eq:sewing} exactly.
\end{lemma}

\begin{proof}
If $V$ is smooth and sewn, then $f=\cZtwo^{-1}V$ is Schwartz
\cite[Theorem~8.2.5]{Grochenig}.  Approximate $f$ in the Schwartz topology by
finite Hermite sums $f_N$.  Each $f_N$ is a polynomial times a Gaussian, so
its Zak series and all mixed derivatives converge normally on compact subsets
of $\C^2$; thus $\cZtwo f_N$ is entire.  The index changes proving
\eqref{eq:sewing} are exact, and continuity of
$\cZtwo:\cS(\R)\to C^\infty([0,1]^2;\C^2)$ gives the asserted convergence.
\end{proof}

\begin{proposition}[Holomorphic invariant gauge]\label{prop:holgauge}
There are $\rho_0>0$, a holomorphic line field $E^s_{\rm hol}$ on
$\Omega_{\rho_0}$, and a nowhere-zero holomorphic vector-Zak section
$v_{\rm an}$ such that
\begin{equation}\label{eq:holinv}
 E^s_{\rm hol}|_{\R^2}=E^s,\qquad
 v_{\rm an}(z)\in E^s_{\rm hol}(z),\qquad
 B(z)E^s_{\rm hol}(Tz)=E^s_{\rm hol}(z).
\end{equation}
\end{proposition}

\begin{proof}
Choose, by \cref{lem:approx}, an entire sewn section $\chi$ arbitrarily close
to $v_{\rm sm}$ on the real torus, and define the holomorphic covector
\begin{equation}\label{eq:ell}
 \ell(z)=\overline{\chi(\overline z)}^{\,T}.
\end{equation}
Since $\overline{U_j(\overline z)}^{T}=U_j(z)^{-1}$, it satisfies the dual
sewing law $\ell(z+e_j)=\ell(z)U_j(z)^{-1}$.  After decreasing a preliminary
strip, $\ell\chi$ is nonzero; set $v_0=\chi/(\ell\chi)$.  On the real torus,
normalize the stable vector by
$w_s=v_{\rm sm}/(\ell v_{\rm sm})$.  The choice of $\chi$ makes
$\delta:=\norm{v_0-w_s}_\infty$ arbitrarily small.

Let $d_p$ be the projective slope metric in the stable tube from
\cref{thm:realinput}.  On a fixed compact sub-tube, normalization by
$\ell v=1$ gives affine coordinates whose norm is uniformly equivalent to
$d_p$: there are $c,C>0$ such that
$c\norm{v-v'}\le d_p(\C v,\C v')\le C\norm{v-v'}$.  Choose $M$ so that
$\kappa_0=(C/c)\theta^M<1/2$ and put $N=16M$.  If
$K_N(z)=B_N(R^Nz)^{-1}$, define
\begin{equation}\label{eq:graphmap}
 \Phi_z(v)=\frac{K_N(z)v}{\ell(z)K_N(z)v},
 \qquad (\Gamma v)(z)=\Phi_z(v(R^Nz)),
\end{equation}
where the input vector lies in the affine fiber
$\ell(R^Nz)v=1$.
On the real torus, $\Gamma w_s=w_s$ and $\Gamma$ is $\kappa_0$-Lipschitz in
these affine coordinates.  Choose a radius $r>0$ inside the affine chart and
then choose $\chi$ so close that
\begin{equation}\label{eq:deltachoice}
 \delta<\frac{1-\kappa_0}{1+\kappa_0}r.
\end{equation}
For $\norm{v-v_0}_\infty\le r$ one then has
\[
 \norm{\Gamma v-v_0}_\infty
 \le\kappa_0(r+\delta)+\delta<r.
\]
Thus the real graph transform maps the closed tube strictly into itself.

The denominator in \eqref{eq:graphmap}, the strict image inclusion, and the
fiber derivative bound have positive margins on a compact real fundamental
domain times the closed fiber ball.  By \eqref{eq:rhodet} and holomorphic
continuity, the same margins persist on a sufficiently thin closed complex
tube $Q_{\rho_0}$, with contraction factor some $\kappa<1$.  Consider the
complete supremum-norm space of compatible sections on $Q_{\rho_0}$ that are
continuous up to the boundary, holomorphic inside, exactly sewn, normalized
by $\ell v=1$, and contained in the radius-$r$ tube.  Block covariance
obtained from \eqref{eq:covariance}, together with the dual sewing of $\ell$,
shows that $\Gamma$ preserves this space.  Banach's theorem gives a unique
fixed section $v_{\rm an}$; uniform convergence of the iterates makes it
holomorphic.  On the real slice uniqueness identifies its line with $E^s$.

The resulting line is $N$-step invariant.  The holomorphic line
$B(z)E^s_{\rm hol}(Tz)$ agrees with $E^s_{\rm hol}(z)$ on the real slice by
the one-step invariance in \cref{thm:realinput}.  In a local affine chart,
the difference of their coordinates is holomorphic and vanishes on
$\R^2$.  Applying the one-variable identity theorem successively in $x$ and
$\omega$, followed by analytic continuation on connected
$\Omega_{\rho_0}$, proves the one-step identity in \eqref{eq:holinv}.
\end{proof}

\begin{proposition}[Holomorphic multiplier removal]\label{prop:remove}
For some $0<\rho_2<\rho_0$ there are a nowhere-zero holomorphic vector-Zak
section $\mathcal G$ on $\Omega_{\rho_2}$ and $\lambda\in\C^*$ such that
\begin{equation}\label{eq:holeigen}
 B(z)\mathcal G(Tz)=\lambda\mathcal G(z).
\end{equation}
\end{proposition}

\begin{proof}
By \cref{prop:holgauge},
\begin{equation}\label{eq:qan}
 B(z)v_{\rm an}(Tz)=q_{\rm an}(z)v_{\rm an}(z)
\end{equation}
for a unique nowhere-zero holomorphic scalar.  Twisted covariance and exact
sewing make $q_{\rm an}$ periodic.  On the real torus,
$v_{\rm an}=a v_{\rm sm}$ for a periodic $a:\T^2\to\C^*$, whence
\begin{equation}\label{eq:coboundary}
 q_{\rm an}(z)=q_{\rm sm}(z)\frac{a(Tz)}{a(z)}.
\end{equation}
The quotient is a multiplicative coboundary, so $q_{\rm an}$ and
$q_{\rm sm}$ have the same two winding numbers; these vanish by
\cref{thm:realinput}.

On a smaller simply connected strip, write $q_{\rm an}=\e^\phi$.  Periodicity
of $q_{\rm an}$ implies that $\phi(z+e_j)-\phi(z)$ is a constant in
$2\pi\ii\Z$; zero winding makes both constants zero, so $\phi$ is periodic.
Let $\phi_0=\int_{\T^2}\phi$ and define, for $(m,n)\ne(0,0)$,
\begin{equation}\label{eq:ucoeff}
 \widehat u(m,n)=\frac{\widehat\phi(m,n)}
 {1-\e^{-2\pi\ii(m\alpha+n\beta)}},\qquad \widehat u(0,0)=0.
\end{equation}
Strip analyticity gives
$\abs{\widehat\phi(m,n)}\le C_r\e^{-2\pi r(\abs m+\abs n)}$ on every
smaller width $r$.  The polynomial loss in \eqref{eq:small} therefore leaves
normal convergence on every still smaller strip.  Hence $u$ is periodic and
holomorphic there and
$u(z)-u(Tz)=\phi(z)-\phi_0$.  Setting
$h=\e^u$, $\lambda=\e^{\phi_0}$, and
$\mathcal G=h v_{\rm an}$ gives \eqref{eq:holeigen}.
\end{proof}

\section{Gelfand--Shilov reconstruction and optimality}

\begin{lemma}[Holomorphic Zak data imply $\cS_1^1$]\label{lem:GS}
If a vector-Zak section $F$ is holomorphic on $\Omega_\rho$ for some
$\rho>0$ and $f=\cZtwo^{-1}F$, then there are $C,H>0$ such that
\begin{equation}\label{eq:factorial}
 \sup_{t\in\R}\abs{t^m f^{(n)}(t)}\le CH^{m+n}m!n!,
 \qquad m,n\ge0.
\end{equation}
\end{lemma}

This strip-holomorphic reconstruction principle was also isolated and proved,
with the same factorial endpoint, by Jasper and Mixon
\cite[Lemma~3]{JasperMixon2026}.  We include the proof to keep the present
argument self-contained and to record the normalization used by
\eqref{eq:vectorzak}.

\begin{proof}
Recover the scalar Zak transform by
$G(x,\nu)=\sqrt2[F(x,2\nu)]_1$.  It is holomorphic for
$\abs{\Im x}<\rho$ and $\abs{\Im\nu}<\rho/2$.  Choose $a,c>0$ with
$2a<\rho$ and $c<\rho/2$.  Cauchy's estimate on a smaller closed tube gives
\begin{equation}\label{eq:Cauchy}
 \abs{\partial_x^nG(x,\nu\pm\ii c)}\le M a^{-n}n!
 \qquad(0\le x,\nu\le1).
\end{equation}
For $x\in[0,1]$ and $k\in\Z$, Fourier inversion gives
\begin{equation}\label{eq:recover}
 f^{(n)}(x-k)=\int_0^1\partial_x^nG(x,\nu)
 \e^{-2\pi\ii k\nu}\,d\nu.
\end{equation}
Move the contour to $\Im\nu=-c$ if $k>0$ and to $\Im\nu=c$ if $k<0$; the
vertical sides cancel by periodicity.  With $d=2\pi c$,
\begin{equation}\label{eq:expdecay}
 \abs{f^{(n)}(x-k)}\le M a^{-n}n!\e^{-d\abs k}.
\end{equation}
Writing $t=x-k$ and using
$\sup_{r\ge0}r^m\e^{-dr}\le d^{-m}m!$ proves
\eqref{eq:factorial}, after enlarging the constant on the central cell.
\end{proof}

\begin{proof}[Proof of \cref{thm:main}]
Take $\mathcal G$ and $\lambda$ from \cref{prop:remove} and set
$g=\cZtwo^{-1}\mathcal G$.  The section is nowhere zero, so $g\ne0$; by
\cref{lem:GS}, $g\in\cS_1^1(\R)$.  The conjugacy
\eqref{eq:conjugacy} yields \eqref{eq:eigen}.  Weyl multiplication gives
\[
 \rho(1,0)\rho(\zeta)=\e^{-\pi\ii\beta/2}
 \rho(1+\alpha,\beta/2),\qquad
 \rho(0,1/2)\rho(\zeta)=\e^{\pi\ii\alpha/2}
 \rho\bigl(\alpha,(1+\beta)/2\bigr),
\]
which expands \eqref{eq:eigen} into \eqref{eq:dependence}.  The geometric
claims follow from \cref{lem:geometry}.
\end{proof}

\begin{proposition}[Obstruction below the endpoint]\label{prop:below}
Let $0\ne f\in\cS_s^\sigma(\R)$.  If $s<1$ or $\sigma<1$, every finite
collection of distinct Weyl shifts of $f$ is linearly independent.
\end{proposition}

\begin{proof}
The standard Gelfand--Shilov characterization gives
\begin{equation}\label{eq:GSdecay}
 \abs{f(x)}\le C\e^{-a\abs{x}^{1/s}},\qquad
 \abs{\widehat f(\xi)}\le C\e^{-b\abs{\xi}^{1/\sigma}}
\end{equation}
for suitable positive constants \cite{Petersson,Toft}.  If $s<1$, then for
every $c>0$,
$-a x^{1/s}+cx\log x\to-\infty$ as $x\to+\infty$.  Thus
\[
 \lim_{x\to+\infty}\abs{f(x)}\e^{cx\log x}=0
 \qquad\text{for every }c>0.
\]
The one-sided theorem of Bownik and Speegle gives the desired independence
\cite{BownikSpeegle}.  If $\sigma<1$, apply the same argument to
$\widehat f\in\cS_\sigma^s$ and use
$\mathcal F\rho(x,\omega)\mathcal F^{-1}=\rho(\omega,-x)$.
\end{proof}

\begin{proof}[Proof of \cref{cor:sharp}]
Systems with at most three distinct phase-space points are independent for
every nonzero $L^2$ window \cite{HRT,Heil}.  Hence \cref{thm:main} reaches
the least possible cardinality.  The exclusion when $s<1$ or $\sigma<1$ is
\cref{prop:below}.  If $s,\sigma\ge1$, then
$\cS_1^1\subseteq\cS_s^\sigma$ directly from \eqref{eq:GSdef}; the window of
\cref{thm:main} therefore supplies a counterexample throughout the closed
upper-right quadrant.
\end{proof}

\begin{remark}
No analytic approximation of the matrix symbol is made.  Such an
approximation would generally create extra Fourier modes and hence additional
time--frequency shifts.  Only an auxiliary center for the graph transform is
approximated; its fixed point is an invariant line for the original
three-term cocycle.  The four-point support is therefore preserved exactly.
\end{remark}

\section*{Acknowledgments}

The author is deeply grateful to Markus Faulhuber, Philipp Petersen,
Jordy Timo van Velthoven, and Felix Voigtlaender for the breakthrough
counterexample in \cite{FPPV2026}; its vector-Zak and cohomological
architecture made the present support compression possible.  The author also
warmly thanks John Jasper and Dustin G. Mixon for their elegant and illuminating
exponential-tail construction \cite{JasperMixon2026}, which established the
$\cS_1^1$ endpoint in a broad Diophantine setting and clarified the regularity
context in which the present support-sharp refinement belongs.  The author is
particularly grateful to Dustin G. Mixon for drawing attention to the endpoint
regularity question and thereby helping the author recognize that the
four-point construction itself admits the sharp Gelfand--Shilov refinement
proved here.

This material is based upon work supported by the National Science
Foundation under Award DMS-2205852, \emph{Collaborative Research: Topics in
Abstract, Applied, and Computational Harmonic Analysis}.

During preparation of the manuscript, ChatGPT (OpenAI) was used to assist
with proofreading, cross-reference and consistency checks, and adversarial
review of the arguments. No language-model output was used as a mathematical
premise or as a substitute for proof. Responsibility for the verification,
final claims, and submitted text remains solely with the author.

\end{document}